\documentclass[12pt]{article}%
\usepackage{amsmath}
\usepackage{amsfonts}
\usepackage{amssymb}
\usepackage{graphicx}%
\providecommand{\U}[1]{\protect\rule{.1in}{.1in}}
\newtheorem{theorem}{Theorem}

\newtheorem{corollary}[theorem]{Corollary}

\newtheorem{example}[theorem]{Example}

\newtheorem{lemma}[theorem]{Lemma}

\newenvironment{proof}[1][Proof]{\noindent\textbf{#1.} }{\ \rule{0.5em}{0.5em}}
\begin{document}

\title{Efficient vectors for block perturbed consistent matrices}
\author{Susana Furtado\thanks{Email: sbf@fep.up.pt. The work of this author was
supported by FCT- Funda\c{c}\~{a}o para a Ci\^{e}ncia e Tecnologia, under
project UIDB/04721/2020.} \thanks{Corresponding author.}\\CEAFEL and Faculdade de Economia \\Universidade do Porto\\Rua Dr. Roberto Frias\\4200-464 Porto, Portugal
\and Charles R. Johnson \thanks{Email: crjohn@wm.edu. The work of this author was
supported in part by National Science Foundation grant DMS-0751964.}\\Department of Mathematics\\College of William and Mary\\Williamsburg, VA 23187-8795.}
\maketitle

\begin{abstract}
In prioritization schemes, based on pairwise comparisons, such as the
Analytical Hierarchy Process, it is important to extract a cardinal ranking
vector from a reciprocal matrix that is unlikely to be consistent. It is
natural to choose such a vector only from efficient ones. Recently a method to
generate inductively all efficient vectors for any reciprocal matrix has been
discovered. Here we focus upon the study of efficient vectors for a reciprocal
matrix that is a block perturbation of a consistent matrix in the sense that
it is obtained from a consistent matrix by modifying entries only in a proper
principal submatrix. We determine an explicit class of efficient vectors for
such matrices. Based upon this, we give a description of all the efficient
vectors in the 3-by-3 block perturbed case. In addition, we give sufficient
conditions for the right Perron eigenvector of such matrices to be efficient
and provide examples in which efficiency does not occur. Also, we consider a
certain type of constant block perturbed consistent matrices, for which we may
construct a class of efficient vectors, and demonstrate the efficiency of the
Perron eigenvector. Appropriate examples are provided throughout.

\end{abstract}

\textbf{Keywords}: consistent matrix, decision analysis, efficient vector,
Perron eigenvector, reciprocal matrix, strongly connected digraph.

\textbf{MSC2020}: 90B50, 91B06, 05C20, 15B48, 15A18

\bigskip

\section{Introduction\label{secintr}}

\bigskip In the Analytic Hierarchy process, a method introduced by Saaty
\cite{saaty1977} and used in decision analysis, reciprocal matrices, also
called pairwise comparison matrices, appear in the context of independent,
pairwise ratio comparisons among $n$ alternatives. An $n$-by-$n$ entry-wise
positive matrix $A=\left[  a_{ij}\right]  $ is called \emph{reciprocal} if,
for all $1\leq i,j\leq n$, $a_{ji}=\frac{1}{a_{ij}}.$ We denote by
$\mathcal{PC}_{n}$ the set of all such matrices. Matrix $A$ is said to be
\emph{consistent} if $a_{ij}a_{jk}=a_{ik}$ for all $i,j,k.$ This is the case
if and only if there is a positive vector $w=\left[
\begin{array}
[c]{ccc}%
w_{1} & \ldots & w_{n}%
\end{array}
\right]  ^{T}$ such that $a_{ij}=\frac{w_{i}}{w_{j}}$ for all $i,j.$ Such
vector is unique up to a factor of scale and ranks the different alternatives.
Any matrix in $\mathcal{PC}_{2}$ is consistent. When $n>2,$ consistency of the
ratio comparisons is unlikely. Thus, a cardinal ranking vector should be
obtained from a reciprocal matrix \cite{choo, dij, golany} and it is natural
to choose it from among efficient ones.

A positive vector $w$ is called \emph{efficient} for $A$ \cite{blanq2006} if,
for every other positive vector $v=\left[
\begin{array}
[c]{ccc}%
v_{1} & \ldots & v_{n}%
\end{array}
\right]  ^{T},$
\[
\left\vert A-\left[  \frac{v_{i}}{v_{j}}\right]  \right\vert \leq\left\vert
A-\left[  \frac{w_{i}}{w_{j}}\right]  \right\vert
\]
(entry-wise) implies $v=w,$ i.e. no other consistent matrix approximating $A$
is clearly better than the one associated with $w$ (Pareto optimality). We
denote the set of all efficient vectors for $A\ $ by $\mathcal{E}(A).$
Clearly, any positive multiple of an efficient vector for $A$ is still efficient.

\bigskip

The efficient vectors for a consistent matrix are the positive multiples of
any of its columns. When $A$ is not consistent, there are infinitely many
(non-proportional) efficient vectors for $A.$ It is known that a vector $w$ is
efficient for $A$ if and only if a certain directed graph (digraph) $G(A,w),$
constructed from $A$ and $w,$ is strongly connected \cite{blanq2006}. Many
ways of constructing efficient vectors for $A$ have been proposed. For
example, the (Hadamard) geometric mean of all the columns of a reciprocal
matrix is efficient \cite{blanq2006}. More recently, it has been proved that
the geometric mean of any subset of the columns is efficient \cite{FJ1}. The
classical proposal for the ranking vector obtained from a reciprocal matrix is
its (right) Perron eigenvector \cite{saaty1977, Saaty}. It is known that this
vector may not be efficient \cite{blanq2006, bozoki2014}. Classes of matrices
for which the Perron eigenvector is efficient have been identified \cite{p6,
p2, FerFur}. Generalizing these results, all the efficient vectors in some of
these classes have been described. More precisely, in \cite{CFF} the
description of the efficient vectors for a matrix obtained from a consistent
matrix by perturbing one entry above the main diagonal and the corresponding
reciprocal entry has been presented. These matrices were called in \cite{p6},
where the efficiency of the Perron eigenvector was shown, simple perturbed
consistent matrices. Later, in \cite{Fu22}, the characterization of the
efficient vectors of a double perturbed consistent matrix, that is, a
reciprocal matrix obtained from a consistent one by modifying at most two
entries above de main diagonal, and the corresponding reciprocal entries, was
presented, generalizing the result that the Perron eigenvector is efficient
\cite{p2}. In \cite{FerFur} sufficient conditions for the efficiency of the
Perron eigenvector of a reciprocal matrix obtained from a consistent one by
perturbing entries located in a $2$-by-$2$ submatrix not containing a diagonal
entry, were given. Several other aspects of efficiency have been studied (see,
for example, \cite{anh, baj, european}). Recently a method to generate
inductively all efficient vectors for a reciprocal matrix was provided
\cite{FJ2}.

\bigskip

In this paper we focus on reciprocal matrices $A\in\mathcal{PC}_{n}$ obtained
from a consistent matrix by modifying entries in an $s$-by-$s$ principal
submatrix, $s<n$. We call such matrices \emph{s-block perturbed consistent
matrices}. We give a class of efficient vectors for such matrices and use this
result to give an explicit description of the complete set of efficient
vectors when $s=3$. This description uses the characterization of the
efficient vectors for an arbitrary matrix in $\mathcal{PC}_{4}$ given in
\cite{FJ2}. We also study the efficiency of the Perron eigenvector of these
matrices. We finally consider block perturbed consistent matrices in which the
perturbed block has a constant pattern, give a class of efficient vectors for
such matrices and prove the efficiency of the eigenvector.

\medskip

Note that the $2$-block perturbed consistent matrices are precisely the simple
perturbed matrices mentioned above. The class of $3$-block perturbed
consistent matrices contains a certain type of double perturbed consistent matrices.

\bigskip

\bigskip

The paper is organized as follows. In the next section we introduce some
notation and give some useful background that will be important throughout.
Then, in Section \ref{s3}, we give a class of efficient vectors for s-block
perturbed consistent matrices. This class is formed by all the efficient
vectors for which the appropriate subvector is efficient for the perturbed
block. Based on this result, in Section \ref{s5} we give a new short proof of
the characterization of the efficient vectors for a $2$-block perturbed
consistent matrix. In Section \ref{s6} we give an explicit description of the
efficient vectors for an arbitrary $3$-block perturbed consistent matrix and
provide sufficient conditions for the efficiency of the Perron eigenvector.
Several examples in which the Perron eigenvector is not efficient are
presented. In Section \ref{s7} we focus on block perturbed consistent matrices
whose perturbed block has constant entries above (below) the main diagonal. We
construct a class of efficient vectors for such matrices and, in addition,
show that the Perron eigenvector is efficient. We conclude the paper with some
observations in Section \ref{s8}.

\section{Background}

In this section we introduce some notation and give some known results that
will be helpful in the paper.

We denote by $\mathcal{V}_{n}$ the set of positive $n$-vectors. For an
$n$-by-$n$ matrix $A=[a_{ij}],$ the principal submatrix of $A$ determined by
deleting (by retaining) the rows and columns indexed by a subset
$K\subseteq\{1,\ldots,n\}$ is denoted by $A(K)$ $(A[K]);$ we abbreviate
$A(\{i\})$ as $A(i).$ Similarly, if $w\in\mathcal{V}_{n},$ we denote by $w(K)$
($w[K]$) the vector obtained from $w$ by deleting (by retaining) the entries
indexed by $K$ and abbreviate $w(\{i\})$ as $w(i)$. Note that, if $A$ is
reciprocal (consistent) then so are $A(K)$ and $A[K].$

We denote by $J_{m,n}$ the $m$-by-$n$ matrix with all entries equal to $1.$
When $m=n$ we simply write $J_{n}.$ We denote by $I_{n}$ the identity matrix
of order $n.$ In all cases, when the size of the matrix is clear from the
context, we may omit the subscript.

\bigskip\ 

The next result (see, for example, \textrm{\cite{Fu22}) }concerns how
$\mathcal{E}(A)$ changes when $A$ is subjected to either a positive diagonal
similarity or a permutation similarity, or both (a monomial similarity) and
allow us to simplify the proofs of our results.

\begin{lemma}
\label{lsim} Suppose that $A\in\mathcal{PC}_{n}$ and $w\in\mathcal{E}(A).$ If
$D$ is a positive diagonal matrix ($P$ is a permutation matrix), then
$DAD^{-1}\in\mathcal{PC}_{n}$ and $Dw\in\mathcal{E}(DAD^{-1})$ ($PAP^{T}%
\in\mathcal{PC}_{n}$ and $Pw\in\mathcal{E}(PAP^{T})$).
\end{lemma}

\bigskip In \cite{blanq2006} the authors proved a useful result that gives a
characterization of efficiency in terms of a certain digraph. Given
$A\in\mathcal{PC}_{n}$ and a positive vector
\begin{equation}
w=\left[
\begin{array}
[c]{ccc}%
w_{1} & \cdots & w_{n}%
\end{array}
\right]  ^{T},\label{ww}%
\end{equation}
define $G(A,w)$ the digraph with vertex set $\{1,\ldots,n\}$ and a directed
edge $i\rightarrow j,$ $i\neq j,$ if and only if, $\frac{w_{i}}{w_{j}}\geq
a_{ij}$. Notice that there is one edge between any two distinct vertices in at
least one direction. Also observe that, for $K\subseteq\{1,\ldots,n\},$
$G(A[K],w[K])$ is the subgraph of $G(A,w)$ induced by the vertices in $K.$

\begin{theorem}
\textrm{\cite{blanq2006}}\label{blanq} Let $A\in\mathcal{PC}_{n}$ and
$w\in\mathcal{V}_{n}.$ The vector $w$ is efficient for $A$ if and only if
$G(A,w)$ is a strongly connected digraph, that is, for all pairs of vertices
$i,j,$ with $i\neq j,$ there is a directed path from $i$ to $j$ in $G(A,w)$.
\end{theorem}

\bigskip

Any $2$-block perturbed consistent matrix $A$ is monomial similar to a
reciprocal matrix $S(x)$ with all entries equal to $1$ except those in
positions $(1,2)$ and $(2,1),$ which are $x$ and $\frac{1}{x},$ respectively,
for some $x>0$. Lemma \ref{lsim} explains how to obtain the efficient vectors
for $A$ from those for $S(x).$

\begin{theorem}
\cite{CFF}\label{tmain} Let $n\geq3,$ $x>0$ and
\[
S(x)=\left[
\begin{tabular}
[c]{l|l}%
$%
\begin{array}
[c]{cc}%
1 & x\\
\frac{1}{x} & 1
\end{array}
$ & $J_{2,n-2}$\\\hline
$J_{n-2,2}$ & $J_{n-2}$%
\end{tabular}
\ \ \right]  \in\mathcal{PC}_{n}.
\]
Then, a vector $w\in\mathcal{V}_{n}$ as in (\ref{ww}) is efficient for $S(x)$
if and only if%
\begin{equation}
w_{2}\leq w_{3},\ldots,w_{n}\leq w_{1}\leq xw_{2}\text{ }\label{f1}%
\end{equation}
or%
\begin{equation}
w_{2}\geq w_{3},\ldots,w_{n}\geq w_{1}\geq xw_{2}\text{. }\label{f2}%
\end{equation}

\end{theorem}

\bigskip Any matrix $A\in\mathcal{PC}_{3}$ is a simple perturbed consistent
matrix. Thus, from Lemma \ref{lsim} and Theorem \ref{tmain}, we obtain the
characterization of the efficient vectors for $A\in\mathcal{PC}_{3}$.

\begin{corollary}
\label{c3por3}Let
\[
\left[
\begin{array}
[c]{ccc}%
1 & a_{12} & a_{13}\\
\frac{1}{a_{12}} & 1 & a_{23}\\
\frac{1}{a_{13}} & \frac{1}{a_{23}} & 1
\end{array}
\right]  .
\]
Then, $w=\left[
\begin{array}
[c]{ccc}%
w_{1} & w_{2} & w_{3}%
\end{array}
\right]  ^{T}\in\mathcal{V}_{3}$ is efficient for $A$ if and only if%
\[
a_{23}w_{3}\leq w_{2}\leq\frac{w_{1}}{a_{12}}\leq\frac{a_{13}}{a_{12}}%
w_{3}\text{\qquad or \qquad}a_{23}w_{3}\geq w_{2}\geq\frac{w_{1}}{a_{12}}%
\geq\frac{a_{13}}{a_{12}}w_{3}.
\]

\end{corollary}

\bigskip

The next result, presented in \cite{FJ1} (see also \cite{CFF}), gives
necessary and sufficient conditions for an efficient vector for a matrix
$A\in\mathcal{PC}_{n}$ to be an extension of an efficient vector for an
$(n-1)$-by-$(n-1)$ principal submatrix of $A.$ It will be often used in this paper.

\begin{theorem}
\cite{FJ1}\label{thext} Let $A\in\mathcal{PC}_{n},$ $w\in\mathcal{V}_{n}$ and
$k\in\{1,\ldots,n\}$. Suppose that $w(k)\in\mathcal{E}(A(k)).$ Then,
$w\in\mathcal{E}(A)$ if and only if
\[
\min_{1\leq i\leq n,\text{ }i\neq k}\frac{w_{i}}{a_{ik}}\leq w_{k}\leq
\max_{1\leq i\leq n,\text{ }i\neq k}\frac{w_{i}}{a_{ik}}.
\]

\end{theorem}

\bigskip

Before we state the inductive method to construct all efficient vectors for an
arbitrary $A\in\mathcal{PC}_{n}$ \cite{FJ2}, we introduce the following
notation, which will be used throughout the paper. For $A\in\mathcal{PC}_{n}$
and $K\subseteq\{1,\ldots,n\},$ we denote by $\mathcal{E}(A,K)$ the set of
vectors $w$ efficient for $A$ and such that $w[K]$ is efficient for $A[K],$
that is, $w\in\mathcal{E}(A)$ and $w[K]\in\mathcal{E}(A[K]).$ We abbreviate
$\mathcal{E}(A,\{1,\ldots,n\}\backslash\{i\})$ as $\mathcal{E}(A;i).$

In \cite{FJ2} it was shown that any efficient vector for $A\in\mathcal{PC}%
_{n}$ has at least two $(n-1)$-subvectors efficient for the corresponding
principal submatrices of $A$ in $\mathcal{PC}_{n-1}.$ As a consequence, we
obtained the following important result.

\begin{theorem}
\cite{FJ2}\label{cunion}Let $A\in\mathcal{PC}_{n}$, $n\geq4.$ Then,
\[
\mathcal{E}(A)=%
{\displaystyle\bigcup\limits_{i=1}^{n}}
\mathcal{E}(A;i)=%
{\displaystyle\bigcup\limits_{i=1,\text{ }i\neq p}^{n}}
\mathcal{E}(A;i),
\]
for any $p\in\{1,\ldots,n\}.$
\end{theorem}

\section{General block perturbed consistent matrices\label{s3}}

In this section we study the efficient vectors for a reciprocal matrix
obtained from a consistent matrix by modifying entries in an $s$-by-$s$
principal submatrix. As mentioned before, we call such a matrix an $s$-block
perturbed consistent matrix. With a possible permutation similarity, such
matrices are of the form%
\[
A=ww^{(-T)}+\left[
\begin{array}
[c]{cc}%
C-vv^{(-T)} & 0\\
0 & 0
\end{array}
\right]  \in\mathcal{PC}_{n},
\]
for some $C\in\mathcal{PC}_{s}$ and some $w=\left[
\begin{array}
[c]{cccc}%
v^{T} & w_{s+1} & \cdots & w_{n}%
\end{array}
\right]  ^{T}\in\mathcal{V}_{n}$ with $v=\left[
\begin{array}
[c]{ccc}%
w_{1} & \ldots & w_{s}%
\end{array}
\right]  ^{T}\in\mathcal{V}_{s}.$ Here $w^{(-T)}$ and $v^{(-T)}$ denote the
row vectors $\left[
\begin{array}
[c]{ccc}%
\frac{1}{w_{1}} & \ldots & \frac{1}{w_{n}}%
\end{array}
\right]  $ and $\left[
\begin{array}
[c]{ccc}%
\frac{1}{w_{1}} & \ldots & \frac{1}{w_{s}}%
\end{array}
\right]  $, respectively$.$ For $D=\operatorname*{diag}(w_{1},\ldots w_{n}),$
we have%
\[
D^{-1}AD=\left[
\begin{array}
[c]{cc}%
B & J\\
J & J
\end{array}
\right]  ,
\]
for some $B\in\mathcal{PC}_{s}.$ Because of Lemma \ref{lsim}, to study the
efficient vectors for $A$, we may assume $D=I_{n},$ that is,
\begin{equation}
A=A_{n}(B)=\left[
\begin{array}
[c]{cc}%
B & J_{s,n-s}\\
J_{n-s,s} & J_{n-s}%
\end{array}
\right]  \in\mathcal{PC}_{n}. \label{BJ}%
\end{equation}
Note that, for $i=1,\ldots,s,$ we have $A(i)=A_{n-1}(B(i)),$ and, for
$i=s+1,\ldots,n,$ we have $A(i)=A_{n-1}(B).$

\bigskip

We make some observations regarding efficient vectors for block perturbed
consistent matrices.

\begin{lemma}
\label{lpermut}Let $A\in\mathcal{PC}_{n}$ as in (\ref{BJ}), with
$B\in\mathcal{PC}_{s}$ and $n>s.$ If $w$ is efficient for $A$ then, for any
$(n-s)$-by-$(n-s)$ permutation matrix $Q,$ $(I_{s}\oplus Q)w$ is efficient for
$A.$
\end{lemma}

\begin{proof}
If $w$ is efficient for $A,$ by Lemma \ref{lsim}, we have that $(I_{s}\oplus
Q)w$ is efficient for $(I_{s}\oplus Q)A(I_{s}\oplus Q^{T})=A.$
\end{proof}

\bigskip

The proof of the next auxiliary lemma is similar to the one of Lemma 6 (item
1.) in \cite{FJ1} and, therefore, is omitted.

\begin{lemma}
\label{laux}Let $A=A_{n}(B)\in\mathcal{PC}_{n}$ as in (\ref{BJ}), with
$B\in\mathcal{PC}_{s}$ and $n>s.$ Let $w=\left[
\begin{array}
[c]{cccc}%
w_{1} & w_{2} & \cdots & w_{n}%
\end{array}
\right]  ^{T}.$ If $w_{p}=w_{q},$ for some $p,q\in\{s+1,\ldots,n\}$ with
$p\neq q,$ then $w$ is efficient for $A$ if and only if $w(p)$ is efficient
for $A(p)=A_{n-1}(B).$
\end{lemma}

\bigskip

\subsection{A class of efficient vectors}

Here, for $A\in\mathcal{PC}_{n}$ as in (\ref{BJ}), with $B\in\mathcal{PC}%
_{s},$ we give a characterization of the subset of $\mathcal{E}(A)$ formed by
the efficient vectors for $A$ whose subvector formed by the first $s$ entries
is efficient for $B$,
\[
\mathcal{E}(A,\{1,\ldots,s\})=\left\{  w\in\mathcal{E}(A):w[\{1,\ldots
,s\}]\in\mathcal{E}(B)\right\}  .
\]

This characterization will be used in the rest of the paper.

\bigskip

The next lemma shows that, if $w\in\mathcal{E}(A,\{1,\ldots,s\}),$ then $w(i)$
is efficient for $A(i)=A_{n-1}(B),$ for any $i\in\{s+1,\ldots,n\}.$

\begin{lemma}
\label{ls+2}Let $A\in\mathcal{PC}_{n}$ be as in (\ref{BJ}), with
$B\in\mathcal{PC}_{s}$ and $n>s.$ Then, for any $i\in\{s+1,\ldots,n\},$%
\begin{equation}
\mathcal{E}(A,\{1,\ldots,s\})\subseteq\mathcal{E}(A;i). \label{rr}%
\end{equation}

\end{lemma}

\begin{proof}
If $n=s+1$ the result is obvious (in this case $\subseteq$ can be replaced by
$=$). Suppose that $n>s+1.$ Let $w\in\mathcal{E}(A,\{1,\ldots,s\}).$ With a
possible permutation similarity on $A$ via an $n$-by-$n$ permutation matrix
$Q=Q_{1}\oplus Q_{2},$ with $Q_{1}$ of size $s$-by-$s$, and taking into
account Lemma \ref{lsim}, we assume, without loss of generality, that
$w_{1}\geq\cdots\geq w_{s}$ and $w_{s+1}\geq\cdots\geq w_{n}.$

By Theorem \ref{blanq}, $G(A,w)$ and $G(A[\{1,\ldots,s\}],w[\{1,\ldots,s\}])$
are strongly connected. Let $i\in\{s+1,\ldots,n\}.$ We show that the subgraph
$G(A(i),w(i))$ is strongly connected, implying that $w(i)$ is efficient for
$A(i).$ Note that, for $j\in\{1,\ldots,n\},$ since $a_{ij}=1$ and $a_{ji}=1$,
$G(A,w)$ contains the edge $i\rightarrow j$ if and only if $w_{i}\geq w_{j}$
and contains the edge $j\rightarrow i$ if and only if $w_{j}\geq w_{i}.$ In
particular, any edge $i\rightarrow j$ with $j\in\{s+1,\ldots,n\},$ $j>i,$ is
in $G(A,w).$

Since $G(A,w)$ is strongly connected, the edges $1\rightarrow s+1$ and
$n\rightarrow s$ are in $G(A,w),$ that is, $w_{1}\geq w_{s+1}>w_{n}\geq
w_{s}.$ This implies that the edges $1\rightarrow j$ and $j\rightarrow s$ are
in $G(A,w)$ for any $j\in\{s+1,\ldots,n\}.$ Then, the paths%
\[
1\rightarrow s+2\rightarrow\cdots\rightarrow n\rightarrow s,
\]%
\[
1\rightarrow s+1\rightarrow s+2\rightarrow\cdots\rightarrow n-1\rightarrow s,
\]
and
\[
1\rightarrow s+1\rightarrow\cdots\rightarrow i-1\rightarrow i+1\rightarrow
\cdots\rightarrow n\rightarrow s,
\]
if $s+1<i<n,$ are in $G(A,w)$. Since $G(A[\{1,\ldots,s\}],w[\{1,\ldots,s\})$
is strongly connected, it follows that $G(A(i),w(i))$ is strongly connected.
\end{proof}

\bigskip

We note that, when $n>s+1,$ the reverse inclusion in (\ref{rr}) may not hold.

\begin{example}
\label{ex21}Let%
\[
A=\left[
\begin{array}
[c]{ccccccc}%
1 & 2 & 1 & 3 & 1 & 1 & 1\\
\frac{1}{2} & 1 & \frac{1}{4} & 1 & 1 & 1 & 1\\
\frac{1}{1} & 4 & 1 & 2 & 1 & 1 & 1\\
\frac{1}{3} & 1 & \frac{1}{2} & 1 & 1 & 1 & 1\\
1 & 1 & 1 & 1 & 1 & 1 & 1\\
1 & 1 & 1 & 1 & 1 & 1 & 1\\
1 & 1 & 1 & 1 & 1 & 1 & 1
\end{array}
\right]  ,
\]
$s=4$ and $w=\left[
\begin{array}
[c]{ccccccc}%
8 & 2 & 6 & 4 & 7 & 3 & 5
\end{array}
\right]  ^{T}.$ We have $w\in\mathcal{E}(A).$ Moreover, $w\in\mathcal{E}(A;i)$
if and only if $i\in\{5,6\}.$ However, $w[\{1,2,3,4\}]\notin\mathcal{E}%
(A[\{1,2,3,4\}]).$
\end{example}

We now characterize all the extensions of vectors in $\mathcal{E}(B)$ to
vectors in $\mathcal{E}(A).$

\begin{theorem}
\label{lcompl}Let $A\in\mathcal{PC}_{n}$ be as in (\ref{BJ}), with
$B\in\mathcal{PC}_{s}$ and $n>s.$ Let $w\in\mathcal{V}_{n}\ $with
$w[\{1,\ldots,s\}]\in\mathcal{E}(B).$ Then, $w\in\mathcal{E}(A)$ if and only
if%
\begin{equation}
\min_{1\leq i\leq s}w_{i}\leq w_{s+1},\ldots,w_{n}\leq\max_{1\leq i\leq
s}w_{i}. \label{ff1}%
\end{equation}

\end{theorem}

\begin{proof}
The proof is by induction on $n.$ If $n=s+1,$ the result follows from Theorem
\ref{thext}. Suppose that $n>s+1$ and $w[\{1,\ldots,s\}]\in\mathcal{E}(B).$

Suppose that (\ref{ff1}) holds, in order to show that $w\in\mathcal{E}(A)$.
Then, we have%
\begin{equation}
\min_{1\leq i\leq s}w_{i}\leq w_{s+1},\ldots,w_{n-1}\leq\max_{1\leq i\leq
s}w_{i}, \label{ff2}%
\end{equation}
implying, by the induction hypothesis, that $w(n)\in\mathcal{E}(A(n)).$ Note
that $A(n)=A_{n-1}(B).$ Since
\[
\min_{1\leq i\leq n-1}w_{i}\leq\min_{1\leq i\leq s}w_{i}\qquad\ \text{and}%
\qquad\max_{1\leq i\leq s}w_{i}\leq\max_{1\leq i\leq n-1}w_{i},
\]
Condition (\ref{ff1}) implies
\begin{equation}
\min_{1\leq i\leq n-1}w_{i}\leq w_{n}\leq\max_{1\leq i\leq n-1}w_{i}.
\label{gg3}%
\end{equation}
Thus, by Theorem \ref{thext}, $w\in\mathcal{E}(A;n)$.

\medskip

As for the converse, suppose that $w\in\mathcal{E}(A),$ in order to show that
(\ref{ff1}) holds$.$ By Lemma \ref{ls+2}, $w(n)$ is efficient for $A(n).$ By
the induction hypothesis, (\ref{ff2}) holds. By Theorem \ref{thext},
(\ref{gg3}) holds. Then, (\ref{ff1}) follows.
\end{proof}

\bigskip

Theorem \ref{lcompl} allows us to easily construct classes of efficient
vectors for $A$ as in (\ref{BJ}) from efficient vectors for $B.$

\begin{example}
\label{expp}Here we give a class of efficient vectors for the matrix
$A=A_{7}(B)$ as in (\ref{BJ}), with%
\[
B=\left[
\begin{array}
[c]{cccc}%
1 & 1 & \frac{3}{4} & \frac{1}{2}\\
1 & 1 & \frac{1}{4} & 1\\
\frac{4}{3} & 4 & 1 & 2\\
2 & 1 & \frac{1}{2} & 1
\end{array}
\right]  .
\]
Let $D=\operatorname*{diag}(1,2,4,2)$ and $C=D^{-1}BD,$ that is,
\begin{equation}
C=\left[  \allowbreak%
\begin{array}
[c]{cccc}%
1 & 2 & 3 & 1\\
\frac{1}{2} & 1 & \frac{1}{2} & 1\\
\frac{1}{3} & 2 & 1 & 1\\
1 & 1 & 1 & 1
\end{array}
\right]  . \label{CC}%
\end{equation}
The set $\mathcal{E(}C)$ was determined in \cite[Section 5]{FJ2}. In
particular, it is noticed that%
\begin{gather*}
\left\{  \left[
\begin{array}
[c]{cccc}%
w_{1} & 4 & 6 & 5
\end{array}
\right]  ^{T}:5\leq w_{1}\leq18\right\}  \subseteq\mathcal{E}(C),\\
\left\{  \left[
\begin{array}
[c]{cccc}%
15 & w_{2} & 8 & 12
\end{array}
\right]  ^{T}:4\leq w_{2}\leq12\right\}  \subseteq\mathcal{E}(C),
\end{gather*}
By Lemma \ref{lsim},
\[
\mathcal{E(}B)=\{Dw\in\mathcal{V}_{4}:w\in\mathcal{E(}C)\}.
\]
Thus, we get
\begin{gather*}
\left\{  \left[
\begin{array}
[c]{cccc}%
w_{1} & 8 & 24 & 10
\end{array}
\right]  ^{T}:5\leq w_{1}\leq18\right\}  \subseteq\mathcal{E}(B),\\
\left\{  \left[
\begin{array}
[c]{cccc}%
15 & 2w_{2} & 32 & 24
\end{array}
\right]  ^{T}:4\leq w_{2}\leq12\right\}  \subseteq\mathcal{E}(B).
\end{gather*}
By Theorem \ref{lcompl},
\[
\left\{  \left[
\begin{array}
[c]{ccccccc}%
w_{1} & 8 & 24 & 10 & w_{5} & w_{6} & w_{7}%
\end{array}
\right]  ^{T}:5\leq w_{1}\leq18\text{ and }\min\{8,w_{1}\}\leq w_{5}%
,w_{6},w_{7}\leq24\right\}  \subseteq\mathcal{E}(A),
\]%
\[
\left\{  \left[
\begin{array}
[c]{ccccccc}%
15 & 2w_{2} & 32 & 24 & w_{5} & w_{6} & w_{7}%
\end{array}
\right]  ^{T}:4\leq w_{2}\leq12\text{ and }\min\{15,2w_{2}\}\leq w_{5}%
,w_{6},w_{7}\leq32\right\}  \subseteq\mathcal{E}(A).
\]

\end{example}

\subsection{Remarks on the Perron eigenvector}

We next give some observations about the Perron eigenvector of block perturbed
consistent matrices.

We recall that the spectral radius of a positive square matrix $A$ is a simple
eigenvalue of $A\ $and there is a positive associated eigenvector, which is
called the \emph{Perron eigenvector} of $A$ (with a possible normalization as,
for example, having the last entry equal to $1$) \cite{HJ}.

\begin{lemma}
\label{leigenvstructure}Let $A$ be as in (\ref{BJ}), with $n>s+1.$ Then, any
eigenvector of $A$ associated with a nonzero eigenvalue has the last $n-s$
entries equal.
\end{lemma}

\begin{proof}
Let $\lambda$ be a nonzero eigenvalue of $A.$ For
\[
x=\left[
\begin{array}
[c]{c}%
y\\
z
\end{array}
\right]  \in\mathcal{V}_{n},
\]
with $y\in\mathcal{V}_{s},\ $condition $Ax=\lambda x$ is equivalent to%
\begin{equation}
\left\{
\begin{array}
[c]{l}%
(B-\lambda I_{s})y+J_{s,n-s}z=0\\
J_{n-s,s}y+(J_{n-s}-\lambda I_{n-s})z=0
\end{array}
\right.  . \label{syst1}%
\end{equation}
If $z=\left[
\begin{array}
[c]{ccc}%
z_{s+1} & \cdots & z_{n}%
\end{array}
\right]  ^{T},$ and $y_{\operatorname*{sum}}$ and $z_{\operatorname*{sum}}$
are the sum of the entries of $y$ and $z,$ respectively, the second equation
in (\ref{syst1}) is equivalent to%
\[
y_{\operatorname*{sum}}+z_{\operatorname*{sum}}-\lambda z_{i}=0,\text{
}i=s+1,\ldots,n.
\]
Since $\lambda\neq0,$ the claim follows.
\end{proof}

\bigskip

An immediate consequence of Lemmas \ref{laux} and \ref{leigenvstructure} is
the following.

\begin{lemma}
\label{lsubmatrix}Let $A$ be as in (\ref{BJ}) with $n>s$, and $w$ be the
Perron eigenvector of $A.$ Then, $w\in\mathcal{E}(A)$ if and only if
$w[\{1,\ldots,s+1\}]\in\mathcal{E}(A[\{1,\ldots,s+1\}]).$
\end{lemma}

\bigskip

We also note that, if $A\in\mathcal{PC}_{n}$ and $x$ is the Perron eigenvector
of $A,$ then (up to a positive scalar factor) $Px$ is the Perron eigenvector
of $B=PAP^{-1}\in\mathcal{PC}_{n}$, in which $P$ is either a permutation
matrix or a positive diagonal matrix. Thus, by Lemma \ref{lsim}, to study the
efficiency of the Perron eigenvector of $A$ we may consider any matrix
monomial similar to $A.$ \bigskip

\section{2-block perturbed consistent matrices revisited\label{s5}}

We next use Theorems \ref{cunion} and \ref{lcompl} to give a simple proof of
Theorem \ref{tmain}, which describes all the efficient vectors for a $2$-block
perturbed consistent matrix.

\bigskip

\textbf{Proof of Theorem \ref{tmain}}

If $n=3$ the result can be verified by a direct inspection of the possible
strongly connected digraphs $G(A,w)$ and applying Theorem \ref{blanq}$.$ Let
$n>3.$

\bigskip

Suppose that (\ref{f1}) or (\ref{f2}) holds. Since the result is true for
$3$-by-$3$ matrices, it follows that $w[\{1,2,3\}]\in\mathcal{E}%
(A[\{1,2,3\}]).$ If (\ref{f1}) holds, then%
\[
\min\{w_{1},w_{2},w_{3}\}=w_{2}\leq w_{4},\ldots,w_{n}\leq w_{1}=\max
\{w_{1},w_{2},w_{3}\},
\]
implying, by Theorem \ref{lcompl}, that $w\in\mathcal{E}(A)$. Similarly, if
(\ref{f2}) holds, $w\in\mathcal{E}(A).$

\bigskip

Now suppose that $w\in\mathcal{E}(A)$ in order to show that (\ref{f1}) or
(\ref{f2}) holds. By Theorem \ref{cunion} (or, more directly, by \cite[Section
4]{FJ2}), there are $j_{1},j_{2},j_{3}\in\{1,\ldots,n\},$ with $j_{1}%
<j_{2}<j_{3},$ such that $w[\{j_{1},j_{2},j_{3}\}]$ is efficient for
$A[\{j_{1},j_{2},j_{3}\}].$ If $A[\{j_{1},j_{2},j_{3}\}]$ has all entries
equal to $1,$ then, by Theorem \ref{lcompl}, $w$ is a constant vector and
(\ref{f1}) or (\ref{f2}) holds trivially. Otherwise, $j_{1}=1,$ $j_{2}=2$ and
\[
A[\{1,2,j_{3}\}]=\left[
\begin{array}
[c]{ccc}%
1 & x & 1\\
\frac{1}{x} & 1 & 1\\
1 & 1 & 1
\end{array}
\right]  .
\]
Since the result holds for $3$-by-$3$ matrices$,$ we have
\[
w_{2}\leq w_{j_{3}}\leq w_{1}\leq xw_{2}\text{ or }w_{2}\geq w_{j_{3}}\geq
w_{1}\geq xw_{2}.
\]
By Theorem \ref{lcompl}, in the former case (\ref{f1}) holds, and in the
latter case (\ref{f2}) holds.

\bigskip

From Theorems \ref{tmain} and \ref{lcompl}, we get the following.

\begin{corollary}
Let $A\in\mathcal{PC}_{n}$ be as in (\ref{BJ}), with $n\geq4$ and
$B\in\mathcal{PC}_{2}$. Then, $\mathcal{E(}A)=\mathcal{E}(A,\{1,2,j\}),$ for
any $j\in\{3,\ldots,n\}$
\end{corollary}

\section{$3$-block perturbed consistent matrices\label{s6}}

\subsection{Efficient vectors}

Here we give an explicit description of the complete set of efficient vectors
for $A\in\mathcal{PC}_{n}$ as in (\ref{BJ}), with $n\geq5$ and $s=3$. For that
purpose, we will use the description given in \cite{FJ2} of the efficient
vectors for an arbitrary matrix in $\mathcal{PC}_{4}$.

We first give a refinement of the recursive method that follows from Theorem
\ref{cunion} to construct $\mathcal{E}(A),$ for $A$ as above. More precisely,
we show that to determine $\mathcal{E}(A)$ we do not need to calculate
$\mathcal{E}(A;i),$ for $i=1,2,3$.

\begin{theorem}
\label{ln5}Let $A\in\mathcal{PC}_{n}$ as in (\ref{BJ}), with $B\in
\mathcal{PC}_{3}$ and $n\geq5.$ Then,%
\[
\mathcal{E}(A)=\cup_{i=4}^{n}\mathcal{E}(A;i).
\]

\end{theorem}

\begin{proof}
We prove the inclusion $\subseteq,$ as the other one is obvious. Let
$w\in\mathcal{E}(A).$ We show that there is $i\in\{4,\ldots,n\}$ such that
$G(A(i),w(i))$ is strongly connected, implying, by Theorem \ref{blanq}, that
$w(i)$ is efficient for $A(i)$. Taking into account Lemma \ref{lsim} and using
arguments similar to those in the proof of Lemma \ref{ls+2}, we assume,
without loss of generality, that $w_{1}\geq w_{2}\geq w_{3}$ and $w_{4}%
\geq\cdots\geq w_{n}.$ Moreover, by Lemma \ref{laux}, we assume that
$w_{4}>\cdots>w_{n}.$ Then, the path
\[
4\rightarrow\cdots\rightarrow n
\]
is in $G(A,w)$ (and no edge in the reverse direction is in $G(A,w)$)$.$ Since
the graph $G(A,w)$ is strongly connected, the edges
\[
1\rightarrow4\quad\text{ and\quad\ }n\rightarrow3
\]
are in $G(A,w)\ $(that is, $w_{1}\geq w_{4}$ and $w_{n}\geq w_{3}),$ implying
that the edges $1\rightarrow i$ and $i\rightarrow3$ are in $G(A,w)$ for any
$i\in\{4,\ldots,n\}.$

If $G(A\left[  \{1,2,3\}\right]  ,w\left[  \{1,2,3\}\right]  )$ is strongly
connected, the result follows from Lemma \ref{ls+2}. Suppose that $G(A\left[
\{1,2,3\}\right]  ,w\left[  \{1,2,3\}\right]  )$ is not strongly connected.

\medskip

Case 1) Suppose that $3\rightarrow2$ is in $G(A,w).$

If $2\rightarrow1$ is in $G(A,w),$ then the cycle
\[
1\rightarrow5\rightarrow\cdots\rightarrow n\rightarrow3\rightarrow
2\rightarrow1
\]
is in $G(A(4),w(4))$, implying that $G(A(4),w(4))$ is strongly connected.

Suppose that $2\rightarrow1$ is not in $G(A,w).$ Then, $1\rightarrow2$ is in
$G(A,w).$ Since the indegree of vertex $1$ is nonzero, we have that
$3\rightarrow1$ or $4\rightarrow1$ is in $G(A,w).$

Case 1.1) Suppose that $3\rightarrow1$ is in $G(A,w).$ Since the outdegree of
vertex $2$ is nonzero and $G(A\left[  \{1,2,3\}\right]  ,w\left[
\{1,2,3\}\right]  )$ is not strongly connected, the edge $2\rightarrow n$ is
in $G(A,w).$ Then,
\[
1\rightarrow5\rightarrow\cdots\rightarrow n\rightarrow3\rightarrow2\rightarrow
n\rightarrow3\rightarrow1
\]
is in $G(A(4),w(4)),$ implying that $G(A(4),w(4))$ is strongly connected.

Case 1.2) Suppose that $3\rightarrow1$ is not in $G(A,w)$ and $4\rightarrow1$
is in $G(A,w)$ (that is $w_{4}=w_{1}$)$.$ Since the only edge to vertex $1$ is
from vertex $4,$ it follows that there is $i\neq1$ such that $i\rightarrow4$
is in $G(A,w).$ This implies $i=2.$ Then,%
\[
1\rightarrow4\rightarrow\cdots\rightarrow n-1\rightarrow3\rightarrow
2\rightarrow4\rightarrow1
\]
is in $G(A(n),w(n))$, implying that $G(A(n),w(n))$ is strongly connected.

\medskip Case 2) Suppose that $3\rightarrow2$ is not in $G(A,w).$ Then
$2\rightarrow3$ is in $G(A,w).$ Since the outdegree of vertex $3$ is
nonzero$,$ at least one of the edges $3\rightarrow1$ or $3\rightarrow n$ is in
$G(A,w).$

Case 2.1) Suppose that $3\rightarrow1$ is in $G(A,w).$ Since $G(A\left[
\{1,2,3\}\right]  ,w\left[  \{1,2,3\}\right]  )$ is not strongly connected
then $1\rightarrow2$ is not an edge in $G(A,w)$ and, thus, $2\rightarrow1$ is
in $G(A,w).$ Since the indegree of vertex $2$ is nonzero, $4\rightarrow2$ is
in $G(A,w).$ Then,
\[
1\rightarrow4\rightarrow\cdots\rightarrow n-1\rightarrow3\rightarrow
1\rightarrow4\rightarrow2\rightarrow1
\]
is in $G(A(n),w(n))$, implying that $G(A(n),w(n))$ is strongly connected.

Case 2.2) Suppose that $3\rightarrow1$ is not in $G(A,w)$ and $3\rightarrow n$
is in $G(A,w)$ (that is $w_{3}=w_{n}$)$.$ Since the only edge from vertex $3$
is to vertex $n,$ it follows that there is $i\neq3$ such that $n\rightarrow i$
is in $G(A,w).$ This implies $i=2.$ Since the indegree of vertex $1$ is
nonzero and $2\rightarrow4$ is not in $G(A,w),$ we have that $2\rightarrow1$
is in $G(A,w)$ (see Case 1.2))$.$ Thus,%
\[
1\rightarrow5\rightarrow\cdots\rightarrow n\rightarrow3\rightarrow
n\rightarrow2\rightarrow1
\]
is in $G(A(4),w(4))$, implying that $G(A(4),w(4))$ is strongly connected.
\end{proof}

\bigskip

We note that Theorem \ref{ln5} does not hold if $n=4.$ For example, for the
matrix $C$ in Example \ref{expp}, $\left[
\begin{array}
[c]{cccc}%
3 & 2 & 1 & 2
\end{array}
\right]  ^{T}\in\mathcal{E}(C)$ but $\left[
\begin{array}
[c]{ccc}%
3 & 2 & 1
\end{array}
\right]  ^{T}\notin\mathcal{E}(C(4))$, as noticed in \cite{FJ2}.

\bigskip

As a consequence of Theorem \ref{ln5}, we next show that, for any vector $w$
efficient for $A,$ there is some $j\in\{4,\ldots,n\}$ such that
$w[\{1,2,3,j\}]$ is efficient for
\[
A\left[  \{1,2,3,j\}\right]  =\left[
\begin{tabular}
[c]{c|c}%
$B$ & $%
\begin{array}
[c]{c}%
1\\
1\\
1
\end{array}
$\\\hline
$%
\begin{array}
[c]{ccc}%
1 & 1 & 1
\end{array}
$ & $1$%
\end{tabular}
\right]  \in\mathcal{PC}_{4}.
\]

\begin{corollary}
\label{ct}Let $A\in\mathcal{PC}_{n}$ as in (\ref{BJ}), with $B\in
\mathcal{PC}_{3}$ and $n\geq4.$ Then,
\[
\mathcal{E}(A)=\cup_{j=4}^{n}\mathcal{E}(A,\{1,2,3,j\}).
\]

\end{corollary}

\begin{proof}
We prove the inclusion $\subseteq,$ as the other one is clear. The proof is by
induction on $n.$ If $n=4$ the result is clear. Suppose that $n>4.$ Let
$w\in\mathcal{E}(A).$ By Theorem \ref{ln5}, there is $i\in\{4,\ldots,n\}$ such
that $w\in\mathcal{E}(A;i).$ Then $w(i)\in\mathcal{E}(A(i)).$ By the induction
hypothesis, $\mathcal{E}(A(i))=\cup_{j=4}^{n-1}\mathcal{E}(A(i),\{1,2,3,j\}).$
Thus, $w(i)\in\mathcal{E}(A(i),\{1,2,3,j\})$ for some $j\in\{4,\ldots,n-1\},$
implying that $w\in\mathcal{E}(A,\{1,2,3,j^{\prime}\})$, in which $j^{\prime
}=j$ if $j<i,$ and $j^{\prime}=j+1$ if $j>i.$
\end{proof}

\bigskip

In Example \ref{exs3} we can see that the sets $\mathcal{E}(A,\{1,2,3,j\})$
may be different for distinct $j$'s in $\{4,\ldots,n\}.$

\bigskip Taking into account Lemma \ref{lpermut}, it follows from Corollary
\ref{ct} that $\mathcal{E}(A)$ is the set of vectors obtained from vectors in
$\mathcal{E}(A,\{1,2,3,4\})$ by permuting the last $n-3$ entries.

\begin{corollary}
\label{ct2}Let $A\in\mathcal{PC}_{n}$ as in (\ref{BJ}), with $B\in
\mathcal{PC}_{3}$ and $n\geq4.$ Then,
\begin{align*}
\mathcal{E}(A)  &  =\left\{  (I_{3}\oplus P)w:w\in\mathcal{E}%
(A,\{1,2,3,4\})\text{ and}\right. \\
&  \left.  P\text{ is an }(n-3)\text{-by-}(n-3)\text{ permutation
matrix}\right\}  .
\end{align*}

\end{corollary}

\begin{proof}
Since $\mathcal{E}(A)\supseteq\mathcal{E}(A,\{1,2,3,4\})$, the inclusion
$\supseteq$ follows by Lemma \ref{lpermut}. Now we show the other inclusion.
Let $v\in\mathcal{E}(A).$ By Corollary \ref{ct}, there is $j\in\{4,\ldots,n\}$
such that $v[\{1,2,3,j\}]\in\mathcal{E(}A[\{1,2,3,j\}]).$ Let $P$ be an
$(n-3)$-by-$(n-3)$ permutation matrix such that $v$ and $w:=(I_{3}\oplus
P^{T})v$ have the $4$th and $j$th entries interchanged. By Lemma \ref{lsim},
$w$ is efficient for $(I_{3}\oplus P^{T})A(I_{3}\oplus P)=A.$ On the other
hand, $w[\{1,2,3,4\}]=v[\{1,2,3,j\}]$ is efficient for
$A[\{1,2,3,4\}]=A[\{1,2,3,j\}].$ Thus, $w\in\mathcal{E}(A,\{1,2,3,4\}).$ Since
$v=(I_{3}\oplus P)w,$ the claim follows.
\end{proof}

\bigskip

To construct $\mathcal{E}(A),$ with $A$ as in (\ref{BJ}), we first obtain the
set of efficient vectors for $A[\{1,2,3,4\}]\in\mathcal{PC}_{4},$ as described
in \cite{FJ2}. Then, we use Theorem \ref{lcompl} to obtain $\mathcal{E}%
(A,\{1,2,3,4\}).$ Finally, we apply Corollary \ref{ct2} to get $\mathcal{E}%
(A).$ We illustrate this procedure in the next example. (We could also
construct $\mathcal{E}(A)$ using Corollary \ref{ct} instead of Corollary
\ref{ct2}.)

\begin{example}
\label{exs3}Consider%
\[
A=\left[
\begin{array}
[c]{cccccc}%
1 & 2 & 3 & 1 & 1 & 1\\
\frac{1}{2} & 1 & \frac{1}{2} & 1 & 1 & 1\\
\frac{1}{3} & 2 & 1 & 1 & 1 & 1\\
1 & 1 & 1 & 1 & 1 & 1\\
1 & 1 & 1 & 1 & 1 & 1\\
1 & 1 & 1 & 1 & 1 & 1
\end{array}
\right]  \in\mathcal{PC}_{6}.
\]
Note that $A[\{1,2,3,4\}]$ is the matrix (\ref{CC}) in Example \ref{expp}. The
set of efficient vectors for this matrix was obtained in \cite[Section 5]%
{FJ2}. Write $w=\left[
\begin{array}
[c]{ccc}%
w_{1} & \cdots & w_{6}%
\end{array}
\right]  ^{T}.$ By Theorem \ref{lcompl},%
\begin{align*}
\mathcal{E}(A,\{1,2,3,4\})  &  =\left\{  w\in\mathcal{V}_{6}:\left[
\begin{array}
[c]{cccc}%
w_{1} & w_{2} & w_{3} & w_{4}%
\end{array}
\right]  ^{T}\in\mathcal{E}(A[\{1,2,3,4\}])\text{ and}\right. \\
&  \left.  \min\left\{  w_{1},w_{2},w_{3},w_{4}\right\}  \leq w_{5},\text{
}w_{6}\leq\max\left\{  w_{1},w_{2},w_{3},w_{4}\right\}  \right\}  .
\end{align*}
By Corollary \ref{ct2}, $\mathcal{E}(A)$ is the set of vectors obtained from
the vectors in $\mathcal{E}(A,\{1,2,3,4\})$ by considering all the
permutations of the last $n-3$ entries.

For example, $\left[
\begin{array}
[c]{cccc}%
13 & 8 & 7 & 12
\end{array}
\right]  ^{T}\in\mathcal{E}(A,\{1,2,3,4\})$ (see \cite[Section 5]{FJ2})$.$
Thus,
\begin{align*}
\left\{  \left[
\begin{array}
[c]{cccccc}%
13 & 8 & 7 & 12 & w_{5} & w_{6}%
\end{array}
\right]  ^{T}:7\leq w_{5},\text{ }w_{6}\leq13\right\}   &  \subseteq
\mathcal{E}(A),\\
\left\{  \left[
\begin{array}
[c]{cccccc}%
13 & 8 & 7 & w_{4} & 12 & w_{6}%
\end{array}
\right]  ^{T}:7\leq w_{4},\text{ }w_{6}\leq13\right\}   &  \subseteq
\mathcal{E}(A),\\
\left\{  \left[
\begin{array}
[c]{cccccc}%
13 & 8 & 7 & w_{4} & w_{5} & 12
\end{array}
\right]  ^{T}:7\leq w_{4},\text{ }w_{5}\leq13\right\}   &  \subseteq
\mathcal{E}(A).
\end{align*}

Regarding Corollary \ref{ct}, we notice that the sets $\mathcal{E}%
(A,\{1,2,3,j\})$ are different for distinct $j$'s in $\{4,5,6\}.$ To see this,
let
\[
u=\left[
\begin{array}
[c]{cccccc}%
13 & 8 & 7 & 12 & 7 & 7
\end{array}
\right]  ^{T}\text{ and }v=\left[
\begin{array}
[c]{cccccc}%
13 & 8 & 7 & 7 & 12 & 7
\end{array}
\right]  ^{T}.
\]
We have $u\in\mathcal{E}(A,\{1,2,3,4\}),$ but $u\notin\mathcal{E}%
(A,\{1,2,3,j\})$ for $j=5,6.$ Also, $v\in\mathcal{E}(A,\{1,2,3,5\})$ but
$v\notin\mathcal{E}(A,\{1,2,3,j\})$, for $j=4,6.$ We observe, in addition,
that $u[\{1,2,3\}],$ $v[\{1,2,3\}]\notin\mathcal{E}(A[\{1,2,3\}])$.

\medskip We finally note that, if Theorem \ref{cunion} was used to compute
$\mathcal{E}(A),$ we would apply the following more involved equality:%
\[
\mathcal{E}(A)=\cup_{i=1}^{6}\mathcal{E}(A;i).
\]

\end{example}

\bigskip Next we give an example showing that an analog of Theorem \ref{ln5}
(and of Corollaries \ref{ct} and \ref{ct2}) does not hold if $B\in
\mathcal{PC}_{4}$.

\begin{example}
\label{ex20}Let%
\[
A=\left[
\begin{array}
[c]{cccccc}%
1 & 5 & 2 & 3 & 1 & 1\\
\frac{1}{5} & 1 & \frac{1}{2} & 3 & 1 & 1\\
\frac{1}{2} & 2 & 1 & 2 & 1 & 1\\
\frac{1}{3} & \frac{1}{3} & \frac{1}{2} & 1 & 1 & 1\\
1 & 1 & 1 & 1 & 1 & 1\\
1 & 1 & 1 & 1 & 1 & 1
\end{array}
\right]  .
\]
The vector%
\[
w=\left[
\begin{array}
[c]{cccccc}%
8 & 2 & 3 & 4 & 6 & 2
\end{array}
\right]  ^{T}%
\]
is efficient for $A.$ However, $w(i)$ is efficient for $A(i)$ if and only if
$i=3,4.$

\smallskip We note, however, that the vector
\[
v=\left[
\begin{array}
[c]{cccccc}%
8 & 2 & 6 & 4 & 6 & 2
\end{array}
\right]  ^{T}%
\]
is efficient for $A$ and $w(i)$ is efficient for $A(i)$ with $i=3,4,6.$

\medskip By Theorem \ref{cunion}, $\mathcal{E}(A)=\cup_{i=1}^{6}%
\mathcal{E}(A;i).$ For $i=1,2,3,4,$ Corollary \ref{ct2} (or Corollary
\ref{ct}) can be applied to obtain $\mathcal{E}(A(i)),$ and then Theorem
\ref{thext} to get $\mathcal{E}(A;i).$ For $i=5,6,$ Theorem \ref{cunion} can
be applied to get $\mathcal{E}(A(i))$ (recursively), and then Theorem
\ref{thext} to get $\mathcal{E}(A;i).$
\end{example}

Though Theorem \ref{ln5} does not hold for $s=4,$ Example \ref{ex21} shows
that we may have $w\in\mathcal{E}(A)$ and $w(i)\in\mathcal{E}(A(i))$ only if
$i>4.$

\subsection{Efficiency of the Perron eigenvector}

Here we give sufficient conditions for the Perron eigenvector of a $3$-block
perturbed consistent matrix $A$ as in (\ref{BJ}) to be efficient. We also give
examples of such matrices for which the Perron eigenvector is not efficient.

With a possible permutation similarity on $A$ via an $n$-by-$n$ permutation
matrix $Q\oplus I_{n-3},$ with $Q$ of size $3$-by-$3$, and taking into account
Lemma \ref{lsim}, we may, and do, assume that the entry in position $1,3$ of
$A$ is greater than or equal to $1.$

\begin{theorem}
\label{thPerron3block}Let $A\in\mathcal{PC}_{n}$ as in (\ref{BJ}), with
$n\geq4$,
\begin{equation}
B=\left[
\begin{array}
[c]{ccc}%
1 & a_{12} & a_{13}\\
\frac{1}{a_{12}} & 1 & a_{23}\\
\frac{1}{a_{13}} & \frac{1}{a_{23}} & 1
\end{array}
\right]  \in\mathcal{PC}_{3}, \label{BB}%
\end{equation}
and $a_{13}\geq1.$ Let $q=a_{13}-a_{23}a_{12}.$ The Perron eigenvector of $A$
is efficient if one of the following conditions hold:

\begin{enumerate}
\item $a_{12}\geq1,$ $a_{23}\geq1$ and $q\leq0;$

\item $a_{12}\geq1,$ $a_{23}\leq1$ and $q\geq0;$

\item $a_{12}\leq1,$ $a_{23}\geq1$ and $q\geq0.$
\end{enumerate}
\end{theorem}

\begin{proof}
Let $\lambda$ be the Perron eigenvalue of $A$ and $w$ be the Perron
eigenvector. Notice that $w$ is positive and, by Lemma \ref{leigenvstructure},
its last $n-3$ entries are equal. By Lemma \ref{lsubmatrix}, it is enough to
show that $w^{\prime}=w[\{1,2,3,4\}]$ is efficient for $A^{\prime
}=A[\{1,2,3,4\}],$ which is equivalent to show that the digraph $G(A^{\prime
},w^{\prime})$ is strongly connected, by Theorem \ref{blanq}. Recall that
$G(A^{\prime},w^{\prime})$ is the subgraph of $G(A,w)$ induced by the set of
vertices $\{1,2,3,4\}$ and $i\rightarrow j$ is an edge in $G(A,w)$ if and only
if $\frac{w_{i}}{w_{j}}\geq a_{ij}.$

Denote by $r_{i}$ the $i$-th row of $(A-\lambda I_{n})w.$ Since $(A-\lambda
I_{n})w=0,$ we have%
\begin{align*}
r_{1}-r_{4}  &  =0\text{, }r_{4}-r_{3}=0\text{, }r_{2}-r_{4}=0\\
r_{1}-a_{12}r_{2}  &  =0,\text{ }r_{2}-a_{23}r_{3}=0\text{, }r_{1}-a_{13}%
r_{3}=0.
\end{align*}
Letting $w=\left[
\begin{array}
[c]{cccccc}%
w_{1} & w_{2} & w_{3} & w_{4} & \cdots & w_{4}%
\end{array}
\right]  ^{T},$ an easy computation shows that the previous equalities are
equivalent to, respectively,%
\begin{equation}
\lambda(w_{4}-w_{1})+(a_{12}-1)w_{2}+(a_{13}-1)w_{3}=0, \label{e1}%
\end{equation}%
\begin{equation}
\lambda(w_{3}-w_{4})+(1-\frac{1}{a_{13}})w_{1}+(1-\frac{1}{a_{23}})w_{2}=0,
\label{e2}%
\end{equation}%
\begin{equation}
\lambda(w_{4}-w_{2})+\left(  \frac{1}{a_{12}}-1\right)  w_{1}+(a_{23}%
-1)w_{3}=0, \label{e6}%
\end{equation}%
\begin{equation}
\lambda(a_{12}w_{2}-w_{1})+(a_{13}-a_{23}a_{12})w_{3}+(1-a_{12})(n-3)w_{4}=0,
\label{e3}%
\end{equation}%
\begin{equation}
\lambda(a_{23}w_{3}-w_{2})+\left(  \frac{1}{a_{12}}-\frac{a_{23}}{a_{13}%
}\right)  w_{1}+(1-a_{23})(n-3)w_{4}=0, \label{e4}%
\end{equation}%
\begin{equation}
\lambda(a_{13}w_{3}-w_{1})+\left(  a_{12}-\frac{a_{13}}{a_{23}}\right)
w_{2}+(1-a_{13})(n-3)w_{4}=0. \label{e5}%
\end{equation}
Taking into account that, for each of the equalities above to hold, if two
summands in the left hand side of the equality are both nonnegative (or both
nonpositive) the other summand is nonpositive (nonnegative), we conclude the following.

If Condition 1. in the statement holds, it follows from (\ref{e1}),
(\ref{e2}), (\ref{e3}) and (\ref{e4}) that $w_{1}\geq w_{4}$, $w_{4}\geq
w_{3},$ $\frac{w_{2}}{w_{1}}\geq\frac{1}{a_{12}}$ and $\frac{w_{3}}{w_{2}}%
\geq\frac{1}{a_{23}}.$ Thus, the cycle%
\[
1\rightarrow4\rightarrow3\rightarrow2\rightarrow1
\]
is in $G(A^{\prime},w^{\prime}).$

If Condition 2. in the statement holds, it follows from (\ref{e1}),
(\ref{e6}), (\ref{e4}) and (\ref{e5}) that the cycle%
\[
1\rightarrow4\rightarrow2\rightarrow3\rightarrow1
\]
is in $G(A^{\prime},w^{\prime}).$

If Condition 3. in the statement holds, it follows from (\ref{e2}),
(\ref{e6}), (\ref{e3}) and (\ref{e5}) that the cycle%
\[
1\rightarrow2\rightarrow4\rightarrow3\rightarrow1
\]
is in $G(A^{\prime},w^{\prime}).$

In all cases, $G(A^{\prime},w^{\prime})$ is strongly connected.
\end{proof}

\bigskip

We observe that, if $a_{12}=1$ or $a_{23}=1,$ then Theorem
\ref{thPerron3block} gives the known result that the Perron eigenvector is
efficient for $A$ (see \cite{p6, p2, CFF, Fu22}).

\bigskip

In Table \ref{tab1}, we show with examples that, when $a_{12},$ $a_{23}$ and
$q$ do not satisfy any of the conditions $1.$ - $3.$ in Theorem
\ref{thPerron3block}, the Perron eigenvector may be efficient or not. Note
that, since we are assuming $a_{13}\geq1,$ the case $a_{12}\leq1$ and
$a_{23}\leq1,$ with at least one inequality strict, implies $q>0.$%

\begin{table}[] \centering
$%
\begin{tabular}
[c]{|c|c|c|c|c|c|}\hline
$n=6$ & $a_{12}$ & $a_{13}$ & $a_{23}$ & efficiency & cycle in $G(A^{\prime
},w^{\prime})$\\\hline\hline
$%
\begin{array}
[c]{ccc}%
a_{12}\geq1, & a_{23}\geq1, & q>0
\end{array}
$ & $2$ & $8.5$ & $2$ & no & \\\cline{2-6}
& $2$ & $8$ & $2$ & yes & $1\rightarrow4\rightarrow3\rightarrow2\rightarrow
1$\\\hline\hline
$%
\begin{array}
[c]{ccc}%
a_{12}\geq1, & a_{23}\leq1, & q<0
\end{array}
$ & $100$ & $5.9$ & $0.1$ & no & \\\cline{2-6}
& $90$ & $5.9$ & $0.1$ & yes & $1\rightarrow4\rightarrow2\rightarrow
3\rightarrow1$\\\hline\hline
$%
\begin{array}
[c]{ccc}%
a_{12}\leq1, & a_{23}\geq1, & q<0
\end{array}
$ & $0.1$ & $5.9$ & $140$ & no & \\\cline{2-6}
& $0.1$ & $5.9$ & $130$ & yes & $1\rightarrow2\rightarrow4\rightarrow
3\rightarrow1$\\\hline\hline
$%
\begin{array}
[c]{ccc}%
a_{12}\leq1, & a_{23}\leq1, & q>0
\end{array}
$ & $0.5$ & $8$ & $0.4$ & no & \\\cline{2-6}
& $0.5$ & $9$ & $0.4$ & yes & $1\rightarrow2\rightarrow4\rightarrow
3\rightarrow1$\\\hline
\end{tabular}
\ \ \ $%
\caption{\small{Efficiency (or not) of the Perron eigenvector of a matrix as in (\ref{BJ}), with $n=6$ and  $B$ as in (\ref{BB}), when conditions 1. -  3. in Theorem \ref{thPerron3block} do not hold.}}\label{tab1}%
\end{table}%
$\ $

\section{Constant type s-block perturbed consistent matrices \label{s7}}

\bigskip Let $x>0.$ We denote by $C_{n}(x)$ the reciprocal matrix $[a_{ij}%
]\in\mathcal{PC}_{n},$ with $a_{ij}=x$, $j>i.$ For example,%
\[
C_{4}(x)=\left[
\begin{array}
[c]{cccc}%
1 & x & x & x\\
\frac{1}{x} & 1 & x & x\\
\frac{1}{x} & \frac{1}{x} & 1 & x\\
\frac{1}{x} & \frac{1}{x} & \frac{1}{x} & 1
\end{array}
\right]  .
\]

Here we give a class of efficient vectors for an $s$-block perturbed matrix
$A_{n}(B)$ as in (\ref{BJ}), with $B=C_{s}(x)$. We first obtain a class of
efficient vectors for $C_{s}(x)$ formed by vectors whose appropriate
subvectors are efficient for all leading principal submatrices of $C_{s}(x).$
Using similar ideas, other classes could be constructed, namely the one formed
by vectors whose appropriate subvectors are efficient for all trailing
principal submatrices of $C_{s}(x).$ Then we illustrate how to obtain
efficient vectors for $A_{n}(C_{s}(x)),$ $n>s,$ from efficient vectors for
$C_{s}(x),$ using Theorem \ref{lcompl}.

Finally we show that the Perron eigenvector of $A_{n}(C_{s}(x)),$ with $n\geq
s,$ is efficient.

We assume $x\geq1,$ as $A_{n}(C_{s}(x))$ is permutationally similar to
$A_{n}(C_{s}(\frac{1}{x}))$ and, thus, taking into account Lemma \ref{lsim},
for our purposes we may consider $A_{n}(C_{s}(\frac{1}{x}))$ instead of
$A_{n}(C_{s}(x)).$ We notice that, for completeness, we include the case
$x=1,$ though in this case $A_{n}(C_{s}(x))$ is consistent and the results are obvious.

\subsection{Classes of efficient vectors}

\begin{theorem}
\label{effCn}Let $x\geq1,$ $n\geq3$ and $w=\left[
\begin{array}
[c]{ccc}%
w_{1} & \cdots & w_{n}%
\end{array}
\right]  ^{T}$. If
\begin{equation}
w_{3}\leq\frac{w_{1}}{x}\leq w_{2}\leq xw_{3} \label{v3by3}%
\end{equation}
and%
\begin{equation}
\frac{1}{x}\min\left\{  w_{3},\ldots,w_{i-1}\right\}  \leq w_{i}\leq
\frac{w_{1}}{x}\text{, }i=4,\ldots,n, \label{vectextended}%
\end{equation}
then $w$ is efficient for $C_{n}(x).$
\end{theorem}

\begin{proof}
Let $B=C_{n}(x).$ The proof is by induction on $n.$ By Corollary \ref{c3por3},
$w[\{1,2,3\}]\in\mathcal{E}(B[\{1,2,3])$ if and only if (\ref{v3by3}) holds.
Thus, the claim follows if $n=3.$ Now suppose that $n\geq4$ and conditions
(\ref{v3by3}) and (\ref{vectextended}) hold. By the induction hypothesis (note
that, in particular, the inequality in (\ref{vectextended}) holds for
$i=4,\ldots,n-1$), $w(n)$ is efficient for $B(n)=C_{n-1}(x).$ By Theorem
\ref{thext}, $w$ is efficient for $B$ if (and only if)%
\[
\min\left\{  \frac{w_{1}}{x},\ldots,\frac{w_{n-1}}{x}\right\}  \leq w_{n}%
\leq\max\left\{  \frac{w_{1}}{x},\ldots,\frac{w_{n-1}}{x}\right\}  .
\]
which follows from (\ref{vectextended}).
\end{proof}

\bigskip

From Theorems \ref{lcompl} and \ref{effCn}, we obtain efficient vectors for
matrices as in (\ref{BJ}), in which $B$ is a block of the form $C_{s}(x).$

\begin{example}
\label{exeff}By Theorem \ref{effCn},
\[
\left[
\begin{array}
[c]{ccccc}%
7 & 3 & 2 & 1 & w_{5}%
\end{array}
\right]  ^{T}\in\mathcal{E}(C_{5}(3))\text{,\quad for any }\frac{1}{3}\leq
w_{5}\leq\frac{7}{3}%
\]
and%
\[
\left[
\begin{array}
[c]{ccccc}%
7 & 3 & 2 & \frac{7}{3} & w_{5}%
\end{array}
\right]  ^{T}\in\mathcal{E}(C_{5}(3)),\text{\quad for any }\frac{2}{3}\leq
w_{5}\leq\frac{7}{3}.
\]
Let $A=A_{8}(C_{5}(3)).$ From Theorem \ref{lcompl}, we have%
\[
\left[
\begin{array}
[c]{cccccccc}%
7 & 3 & 2 & 1 & w_{5} & w_{6} & w_{7} & w_{8}%
\end{array}
\right]  ^{T}\in\mathcal{E}(A)\text{,}%
\]
for%
\[
\frac{1}{3}\leq w_{5}\leq\frac{7}{3}\quad\text{and}\quad\min\{1,w_{5}\}\leq
w_{6},w_{7},w_{8}\leq7.
\]
Also,%
\[
\left[
\begin{array}
[c]{cccccccc}%
7 & 3 & 2 & \frac{7}{3} & w_{5} & w_{6} & w_{7} & w_{8}%
\end{array}
\right]  ^{T}\in\mathcal{E}(A)\text{,}%
\]
for%
\[
\frac{2}{3}\leq w_{5}\leq\frac{7}{3}\quad\text{and}\quad\min\{2,w_{5}\}\leq
w_{6},w_{7},w_{8}\leq7.
\]

\end{example}

\subsection{Efficiency of the Perron eigenvector}

Let $n>s$ and $x>0.$ Let $A=A_{n}(B)\in\mathcal{PC}_{n}$ be as in (\ref{BJ}),
with $B=C_{s}(x)\in\mathcal{PC}_{s}.$ Here we show that the Perron eigenvector
of $A$ is efficient$.$ Note that, since $C_{s}(x)=DA_{s}(C_{s-1}(x))D^{-1},$
in which $D$ is the $s$-by-$s$ positive diagonal matrix
$D=\operatorname*{diag}(x,\ldots,x,1),$ taking into account Lemma \ref{lsim},
it also follows from our result that the Perron eigenvector of $C_{s}(x)$ is efficient.

\begin{theorem}
Let $x>0.$ The Perron eigenvector of the matrix $A$ as in (\ref{BJ}), with
$n>s$ and $B=C_{s}(x),$ is efficient for $A.$
\end{theorem}

\begin{proof}
We assume $x\geq1.$ Since, by Lemma \ref{leigenvstructure}, the Perron
eigenvector $w$ of $A$ has the last $n-s$ entries equal, we assume that
$w=\left[
\begin{array}
[c]{cccc}%
y^{T} & 1 & \cdots & 1
\end{array}
\right]  ^{T},$ with $y=\left[
\begin{array}
[c]{ccc}%
w_{1} & \cdots & w_{s}%
\end{array}
\right]  ^{T}.$ By Lemma \ref{lsubmatrix}, it is enough to show that
$w^{\prime}=w[\{1,\ldots,s+1\}]$ is efficient for $A^{\prime}=A[\{1,\ldots
,s+1\}].$ We show that the digraph $G(A^{\prime},w^{\prime})$ contains the
cycle
\[
s+1\rightarrow s\rightarrow\cdots\rightarrow2\rightarrow1\rightarrow s+1.
\]
Then, $G(A^{\prime},w^{\prime})$ is strongly connected implying, by Theorem
\ref{blanq}, that $w^{\prime}$ is efficient for $A^{\prime}.$ The claimed
cycle exists in $G(A^{\prime},w^{\prime})$ if and only if $w_{1}\geq1,$ $1\geq
w_{s}$ and $\frac{w_{i}}{w_{i+1}}\leq x$ for $i=1,\ldots,s-1.$

Let $\lambda$ be the Perron eigenvalue of $A.$ We have $(A-\lambda I_{n})w=0$.
Denote by $r_{i}$ the $i$-th row of $(A-\lambda I_{n})w.$

Let $i\in\{1,\ldots,s-1\}.$ We have $(r_{i}-xr_{i+1})w=0$, which is equivalent
to%
\begin{gather*}
\left(  \frac{1}{x}-1\right)  (w_{1}+\cdots+w_{i-1})-\lambda w_{i}+\lambda
xw_{i+1}+(x-x^{2})(w_{i+2}+\cdots+w_{s})+(1-x)(n-s)=0\\
\Leftrightarrow(1-x)\left[  \frac{1}{x}(w_{1}+\cdots+w_{i-1})+x(w_{i+2}%
+\cdots+w_{s})+n-s\right]  +\lambda(xw_{i+1}-w_{i})=0.
\end{gather*}
Since $1-x\leq0,$ we get $xw_{i+1}-w_{i}\geq0.$

We also have $r_{1}-r_{s+1}=0,$ which is equivalent to%
\[
-\lambda w_{1}+(x-1)(w_{2}+\cdots+w_{s})+\lambda=0\Leftrightarrow
(x-1)(w_{2}+\cdots+w_{s})+\lambda(1-w_{1})=0.
\]
Since $x-1\geq0,$ we get $1-w_{1}\leq0.$

Finally, we have $r_{s}-r_{s+1}=0,$ which is equivalent to%
\[
\left(  \frac{1}{x}-1\right)  (w_{1}+\cdots+w_{s-1})-\lambda w_{s}%
+\lambda=0\Leftrightarrow\left(  \frac{1}{x}-1\right)  (w_{1}+\cdots
+w_{s-1})+\lambda(1-w_{s})=0.
\]
Since $\frac{1}{x}-1\leq0,$ we get $1-w_{s}\geq0.$
\end{proof}

\bigskip

\section{Conclusions\label{s8}$\allowbreak$}

In this paper we focus on reciprocal matrices obtained from consistent ones by
modifying entries in a principal $s$-by-$s$ submatrix, which we call $s$-block
perturbed matrices. We give a class of efficient vectors for such matrices.
Based on it, we provide an explicit description of the complete set of
efficient vectors when $s=3$. This generalizes the previous study of the
efficient vectors for simple and some double perturbed consistent matrices. We
have shown, with an example, that similar results do not hold when $s=4.$ In
addition, sufficient conditions for the efficiency of the Perron eigenvector
of $3$-block perturbed matrices are provided, and examples in which the Perron
eigenvector is not efficient are presented.

Efficient vectors for block perturbed consistent matrices with a constant
pattern perturbation are also given and the efficiency of the Perron
eigenvector is shown.

Several examples illustrating the theoretical results are provided throughout
the paper.

Our results contribute to the study of efficient vectors for reciprocal
matrices, in particular the efficiency of the Perron eigenvector.

\bigskip

\textbf{Declaration} All authors declare that they have no conflicts of interest.

\end{document}